\documentclass[reqno]{amsart}
\usepackage{hyperref}

\AtBeginDocument{{\noindent\small
\emph{Electronic Journal of Differential Equations},
Vol. 2016 (2016), No. 67, pp. 1--17.\newline
ISSN: 1072-6691. URL: http://ejde.math.txstate.edu or http://ejde.math.unt.edu
\newline ftp ejde.math.txstate.edu}
\thanks{\copyright 2016 Texas State University.}
\vspace{8mm}}

\begin{document}
\title[\hfilneg EJDE-2016/67\hfil Nonlinear elliptic equations]
{Nonlinear elliptic equations and systems with linear part at resonance}

\author[P. Korman \hfil EJDE-2016/67\hfilneg]
{Philip Korman}

\address{Philip Korman   \newline
Department of Mathematical Sciences,
University of Cincinnati,
Cincinnati Ohio 45221-0025, USA}
\email{kormanp@ucmail.uc.edu}

\thanks{Submitted  March 27, 2015. Published March 10, 2016.}
\subjclass[2010]{35J61, 35J47}
\keywords{Elliptic systems at resonance; resonance at multiple eigenvalues;
 \hfill\break\indent Lazer and Leach condition}

\begin{abstract}
 The famous result of Landesman and Lazer \cite{L} dealt with resonance
 at a simple eigenvalue. Soon after publication of \cite{L},  Williams \cite{W}
 gave an extension for repeated eigenvalues. The conditions in  Williams \cite{W}
 are rather restrictive, and no examples were ever given. We show that seemingly
 different classical result by Lazer and Leach \cite{L1}, on forced harmonic
 oscillators at resonance, provides an example for this theorem.
 The article by Williams \cite{W} also contained a shorter proof.
 We use a similar approach to study resonance for $2 \times 2$ systems.
 We derive conditions for existence of solutions, which turned out to
 depend on the spectral properties of the linear part.
\end{abstract}

\maketitle
\numberwithin{equation}{section}
\newtheorem{theorem}{Theorem}[section]
\newtheorem{lemma}[theorem]{Lemma}
\allowdisplaybreaks

\section{Introduction}

The famous article by Landesman and  Lazer \cite{L} considered a
 semilinear Dirichlet  problem at resonance (on a bounded domain $D \subset \mathbb{R}^n$)
\begin{equation} \label{i1}
\Delta u+\lambda_k u+g(u)=f(x), \quad x \in D, \quad  u=0 \quad \text{on } \partial D\,.
\end{equation}
Assuming that $g(u)$ has finite limits at $\pm \infty$, and $\lambda _k$ is a
simple eigenvalue of $-\Delta$, they gave a necessary and sufficient
condition for the existence of solutions. This nonlinear version of the
Fredholm alternative has generated an enormous body of research, and perhaps
it can be seen as the beginning of the modern theory of nonlinear PDE's.
Soon after publication of \cite{L}, a more elementary proof  was given by
Williams \cite{W}. Both \cite{L} and \cite{W} were  using Schauder's fixed
point theorem.
 Williams \cite{W} has  observed that one can also handle the case of multiple
eigenvalues under a straightforward extension of the  Landesman and Lazer condition.
No examples when this condition can be verified for repeated eigenvalues were
ever given. Our first result is to observe that another famous result   of
Lazer and  Leach \cite{L1}, on forced harmonic oscillators at resonance,
provides an example for this theorem, for the  case of  double eigenvalues
of the periodic problem in one dimension. This  appears to be the only known example
in case $\lambda _k$ is not simple. It relies on some special properties of the
sine and cosine functions. Thus one has a uniform framework for these results,
into which the existence result of  de Figueiredo and   Ni \cite{FN}, its
generalization by  Iannacci et al \cite{INW} (and our two extensions)
also fit in nicely. We review these results, and connect them to the recent
 work of  Korman and Li \cite{KL}, and  Korman \cite{K2}.

We use a similar approach to give a rather complete discussion of
$2 \times 2$ elliptic systems at resonance, in case its linear part has
constant coefficients. Requiring that finite limits at infinity exist,
as in  Landesman and  Lazer, appears to be too restrictive for systems.
Instead, we use more general inequality type conditions, which are rooted
in  Lazer and Leach \cite{L1}. We derive several sufficient conditions
of this  type for systems at resonance, which turned out to depend
on the spectral properties of the linear part.

\section{An exposition and extensions of known results}

Given a bounded domain $D \subset \mathbb{R}^n$, with a smooth boundary, we denote
by $\lambda _k$ the eigenvalues of the Dirichlet problem
\begin{equation} \label{dl}
\Delta u+\lambda u=0, \quad x \in D, \quad u=0 \quad \text{on }\partial D \,,
\end{equation}
and by $\varphi _k(x)$ the corresponding eigenfunctions. For the \emph{resonant} problem
\begin{equation} \label{1}
\Delta u+\lambda_k u=f(x), \quad x \in D, \quad u=0 \quad \text{on }\partial D \,,
\end{equation}
with a given $f(x) \in L^2(D)$, the following well-known
\emph{Fredholm alternative} holds: the problem \eqref{1} has a
solution if and only if
\begin{equation} \label{2}
\int _D f(x) \varphi _k(x) \, dx=0 \,.
\end{equation}
One could expect things to be considerably harder for the nonlinear problem
\begin{equation}
\label{3}
\Delta u+\lambda_k u+g(u)=f(x), \quad x \in D, \quad u=0 \quad \text{on } \partial D \,.
\end{equation}
However, in the classical papers of  Lazer and  Leach \cite{L1}, and
 Landesman and  Lazer \cite{L} an interesting   class of nonlinearities $g(u)$
was identified, for which one still has an analog of the Fredholm alternative.
Namely, one assumes that the finite limits $g(-\infty)$ and $g(\infty)$ exist,
 and
\begin{equation} \label{4}
g(-\infty)<g(u)<g(\infty), \quad \text{for all } u \in R \,.
\end{equation}
From \eqref{3},
\begin{equation} \label{4.1}
\int _D g(u(x)) \varphi _k(x) \, dx=\int _D f(x) \varphi _k(x) \, dx \,,
\end{equation}
which implies, in view of \eqref{4},
\begin{equation} \label{5}
\begin{aligned}
&g(-\infty) \int_{\varphi _k>0} \varphi _k \, dx+g(\infty) \int_{\varphi _k<0} \varphi _k \, dx\\
&< \int _D f(x) \varphi _k \, dx
 <g(\infty) \int_{\varphi _k>0} \varphi _k \, dx+g(-\infty) \int_{\varphi _k<0} \varphi _k \, dx \,.
\end{aligned}
\end{equation}
This is a necessary condition for solvability. It was proved by  Landesman and
 Lazer \cite{L} that this condition is also sufficient for solvability.

\begin{theorem}[\cite{L}]\label{thm:1}
Assume that $\lambda _k$ is a simple eigenvalue, while $g(u) \in C(R)$ satisfies
\eqref{4}. Then for any $f(x) \in L^2(D)$ satisfying \eqref{5},
the problem \eqref{3} has a solution $u(x) \in W^{2,2}(D) \cap W_0^{1,2}(D)$.
\end{theorem}

We shall prove the sufficiency part under  a condition on $g(u)$, which
is more general than \eqref{4}. This condition had originated in  Lazer and
 Leach \cite{L1}, and it turned out to be appropriate when studying systems
(see the next section).
We assume that $g(u) \in C(R)$ is bounded on $R$, and there exist numbers
$c$, $d$, $C$ and $D$, with  $c<d$ and $C<D$, such that
\begin{gather} \label{20}
g(u) > D \quad \text{for $u > d$} \,, \\
\label{21}
g(u) < C \quad \text{for $u < c$} \,.
\end{gather}
Define
\[
L_2=D \int _{\varphi _k>0} \varphi _k \, dx+C \int _{\varphi _k<0} \varphi _k \, dx, \quad
 L_1=C \int _{\varphi _k>0} \varphi _k \, dx+D \int _{\varphi _k<0} \varphi _k \, dx \,.
\]
Observe that $L_2>L_1$, because $D>C$. We shall denote by $\varphi _k^{\perp}$
the subspace of $L^2(D)$, consisting of functions satisfying \eqref{2}.

\begin{theorem}\label{thm:3}
Assume that $\lambda _k$ is a simple eigenvalue, while $g(u) \in C(R)$ is
 bounded on $R$, and  satisfies \eqref{20} and \eqref{21}. Then for any
 $f(x) \in L^{2}(D)$ satisfying
\begin{equation} \label{22}
L_1<\int _D f(x) \varphi _k \, dx<L_2 \,,
\end{equation}
 problem \eqref{3} has a solution $u(x) \in W^{2,2}(D) \cap W_0^{1,2}(D)$.
\end{theorem}

\begin{proof}
Normalize $\varphi _k(x)$, so that $\int _D \varphi ^2_k(x) \, dx=1$. Denoting
 $A_k=\int _D f(x) \varphi _k \, dx$, we decompose $f(x)=A_k \varphi _k (x)+e(x)$,
with $e(x) \in \varphi _k^{\perp}$ (where $\varphi _k^{\perp}$ is a subspace of
$L^2(D)$). Similarly, we decompose the solution $u(x)=\xi _k \varphi _k (x)+U(x)$,
 with $U(x) \in \varphi _k^{\perp}$. We rewrite \eqref{4.1}, and then \eqref{3}, as
\begin{gather} \label{6}
\int _D g(\xi _k \varphi _k (x)+U(x)) \varphi _k(x) \, dx=A_k \,, \\
\label{7} \begin{gathered}
  \Delta U+\lambda_k U=-g(\xi _k \varphi _k +U)+\varphi _k \int _D g(\xi _k \varphi _k +U )
\varphi _k\, dx+e, \quad x\in D \\
    U=0 \quad \text{on } \partial D \,.
\end{gathered}
\end{gather}
Equations \eqref{6} and \eqref{7} constitute the classical Lyapunov-Schmidt
 reduction of the problem \eqref{3}. To solve this system, we set up a
 map $T:  (\eta _k,V) \to (\xi _k,U)$, taking the space
$R \times \varphi _k^{\perp}$ into itself, by solving the equation
\begin{equation}\label{8}
\begin{gathered}
 \Delta U+\lambda_k U=-g(\eta _k \varphi _k +V)+\varphi _k \int _D g(\eta _k \varphi _k +V )
\varphi _k \, dx+e, \; x\in D \\
  U=0 \quad \text{on } \partial D \,,
\end{gathered}
\end{equation}
for $U$, and then setting
\begin{equation} \label{8.1}
\xi _k=\eta _k +A_k- \int _D g(\eta _k \varphi _k (x)+U(x)) \varphi _k(x) \, dx \,.
\end{equation}
The right hand side of \eqref{8} is orthogonal to $\varphi _k$, and so by the
Fredholm alternative we can find infinitely many solutions $U=U_0+c \varphi _k$.
Then we can choose $c$, so that $U \in \varphi _k^{\perp}$.

Assume first that $f(x) \in L^{\infty}(D)$. By the elliptic theory,
 we can estimate $\|U\|_{W^{2,p}(D)}$ by the $L^p$ norm of the right hand
 side of \eqref{8} plus $\|U\|_{L^p(D)}$, for any $p>1$. Since the homogeneous
equation, corresponding to \eqref{8}, has only the trivial solution in
$\varphi _k^{\perp}$, the $\|U\|_{L^p(D)}$ term can be dropped, giving us a
uniform estimate of $\|U\|_{W^{2,p}(D)}$. By the Sobolev embedding, for
some constant $c>0$ (for $p$ large enough)
\begin{equation} \label{9}
\|U\|_{L^{\infty}(D)} \leq c \quad \text{uniformly in }
(\eta _k,V) \in R \times \varphi _k^{\perp} \,.
\end{equation}
This implies that if $\eta _k$ is large  and positive, the integral in \eqref{6}
is greater than  $L_2$. When $\eta _k$ is large in absolute value and negative,
the integral in \eqref{6} is smaller than $L_1$. By our condition \eqref{22},
  it follows that for $\eta _k$ large and positive,  $\xi _k<\eta _k$, while
for $\eta _k$ large in absolute value and negative,  $\xi _k>\eta _k$.
Hence, we can find a large $N$, so that if $\eta _k \in (-N,N)$,  then
$\xi _k \in (-N,N)$. The map $T: (\eta _k,V) \to (\xi _k,U)$ is a continuous
and compact map, taking a sufficiently large ball of $R \times \varphi _k^{\perp}$
into itself. By Schauder's fixed point theorem (see e.g.,  Nirenberg \cite{N})
 it has a fixed point, which gives us a solution of the problem \eqref{3}.
(A fixed point of \eqref{8.1} is a solution of \eqref{6}.)

In case $f(x) \in L^{2}(D)$,  a little more care is needed to show that the
integral in \eqref{8.1} is greater (smaller) than $A_k$, for $\eta _k$
positive (negative) and large. Elliptic estimates give us
\begin{equation}
\label{9.1}
\|U\|_{L^{2}(D)} \leq c \quad \text{uniformly in }
(\eta _k,V) \in R \times \varphi _k^{\perp} \,.
\end{equation}
Set $G=\sup _{x \in D, \; u \in R } |g(u) \varphi _k(x)|$, and decompose
\begin{equation}
\label{9.2}
\int _D g(\eta _k \varphi _k (x)+U(x)) \varphi _k(x) \, dx= \int _{\varphi _k >0} g  \varphi _k \, dx+ \int _{\varphi _k <0} g  \varphi _k \, dx \,.
\end{equation}
The first integral we decompose further, keeping the same integrand,
\[
\int _{\varphi _k >0}  \, dx=\int _{0<\varphi _k <\delta}  \, dx
+ \int _{A_2}  \, dx+ \int _{A_3}  \, dx \equiv I_1+I_2+I_3 \,,
\]
with $A_2=(\varphi _k >\delta)  \cap  \left( |U|>\frac{\eta _k \varphi _k}{2} \right)$,
and $A_3=(\varphi _k >\delta)  \cap  \left( |U|<\frac{\eta _k \varphi _k}{2} \right)$.
Given any $\epsilon$, we fix $\delta$ so that the measure of the set
$\{ 0<\varphi _k(x) <\delta \}$ is less than  $\epsilon$. Then
 $|I_1| < \epsilon G$. In $I_2$ we integrate over the set, where
$|U|>\frac{\eta _k \delta}{2}$. Since $U$ is bounded in $L^2$ uniformly in $\eta _k$,
the measure of this set will get smaller than $\epsilon$ for $\eta _k$ large,
 and then $|I_2| < \epsilon G$. In $I_3$ we have $g(u)>D$ for $\eta _k$ large,
and we integrate over the subset of the  domain $D_+=\{x : \varphi _k(x)>0 \}$,
 whose measure is smaller than that of $D_+$ by no more than $2 \epsilon$.
 Hence $I_3>D \int _{\varphi _k >0} \varphi _k(x) \, dx-2 \epsilon G $, and then
\[
\int _{\varphi _k >0} g\varphi _k \, dx>D \int _{\varphi _k >0}  \varphi _k(x) \, dx-4 \epsilon G \,.
\]
 Proceeding similarly with the second integral in \eqref{9.2}, we conclude that
\[
\int _D g(\eta _k \varphi _k (x)+U(x)) \varphi _k(x) \, dx>L_2-8\epsilon G >A_k \,,
\]
if $\epsilon$ is small, which can be achieved for $\eta _k>0$ and large.
Similarly, we show that this integral is smaller than $A_k$ for $\eta _k<0$ and large.
\end{proof}

In case $\lambda  _k$ has a multidimensional eigenspace, a generalization of
 Landesman and  Lazer \cite{L} result  follows by a similar argument,
under a suitable condition. This was observed by S.A. Williams \cite{W},
back in 1970. Apparently no such examples in the PDE case for the
 Williams \cite{W} condition were ever given. We remark that  Williams \cite{W}
contained  a proof of   Landesman and  Lazer \cite{L} result, which is
similar to the one above, see also a recent paper of  Hastings and  McLeod \cite{H}.
Our proof below is a little shorter than in \cite{W}.

\begin{theorem}[\cite{W}] \label{thm:10}
Assume that $g(u)$ satisfies \eqref{4}, $f(x) \in L^2(D)$, while for any $w(x)$
belonging to the eigenspace of $\lambda _k$
\begin{equation} \label{9.4}
\int _D f(x) w(x) \, dx<g(\infty) \int _{w>0} w \, dx
+g(-\infty) \int _{w<0} w \, dx \,.
\end{equation}
Then  problem \eqref{3} has a solution. Condition \eqref{9.4} is also necessary
for the existence of solutions.
\end{theorem}

\begin{proof}
Let $E \subset L^2(D)$ denote the eigenspace of $\lambda _k$, and let $\varphi _1$,
 $\varphi _2,\dots,\varphi _m$ be its orthogonal basis, with
 $\int _D \varphi ^2_i(x) \, dx=1$, $1 \leq i \leq m$.  Denoting
$A_i=\int _D f(x) \varphi _i \, dx$, we decompose
$f(x)=\sum_{i=1}^m  A_i \varphi _i (x)+e(x)$, with $e(x) \in E ^{\perp}$
 (where $E ^{\perp}$ is a subspace of $L^2(D)$). Similarly, we decompose the
solution $u(x)=\sum_{i=1}^m  \xi_i \varphi _i (x)+U(x)$, with $U(x) \in E^{\perp}$.
We have
\begin{gather} \label{9.41}
\int _D g(\sum_{i=1}^m  \xi_i \varphi _i+U(x)) \varphi _i(x) \, dx=A_i,
\quad i=1,\ldots,m \,, \\
\label{9.42}
\begin{gathered}
 \Delta U+\lambda_k U=-g(\sum_{i=1}^m  \xi_i \varphi _i +U)+\sum_{i=1}^m \varphi _i
\int _D g(\sum_{i=1}^m  \xi_i \varphi _i+U ) \varphi _i\, dx+e \\
   U=0 \quad \text{on } \partial D \,.
\end{gathered}
\end{gather}
Equations \eqref{9.41} and \eqref{9.42} constitute the classical
Lyapunov-Schmidt reduction of problem \eqref{3}. To solve this system,
we set up a map $T: (\eta _1,\ldots, \eta _m,V) \to (\xi _1,\ldots, \xi _ m,U)$,
 taking the space $R^m \times E^{\perp}$ into itself, by solving the equation
\begin{equation} \label{9.43}
\begin{gathered}
\Delta U+\lambda_k U=-g(\sum_{i=1}^m  \eta _i \varphi _i +V
)+\sum_{i=1}^m\varphi _i \int _D g(\sum_{i=1}^m  \eta _i \varphi _i +V ) \varphi _i \, dx+e \\
  U=0 \quad \text{on } \partial D \,,
\end{gathered}
\end{equation}
for $U$, and then setting
\begin{equation} \label{9.44}
\xi _i=\eta _i +A_i- \int _D g(\sum_{i=1}^m \eta _i \varphi _i (x)+U(x)) \varphi _i(x)
\, dx, \quad i=1,\ldots,m \,.
\end{equation}
The right-hand side of \eqref{9.43} is orthogonal to all $\varphi _i$, and so by
the Fredholm alternative we can find infinitely many solutions
$U=U_0+\sum_{i=1}^m c_i \varphi _i$. Then we can choose $c_i$, so that $U \in E^{\perp}$.

As before, we get a bound on $\|U\|_{L^{2}(D)}$, uniformly in
$(\eta _1,\ldots, \eta _m,V)$. Denoting
$I_i=\int _D g(\sum_{i=1}^m \eta _i \varphi _i (x)+U(x)) \varphi _i(x) \, dx$, we have
\begin{equation} \label{9.45}
\sum_{i=1}^m \xi _i ^2=\sum_{i=1}^m \eta _i ^2+2\sum_{i=1}^m
\eta _i(A_i-I_i)+\sum_{i=1}^m (A_i-I_i)^2 \,.
\end{equation}
Denoting $w=\sum_{i=1}^m \frac{\eta _i}{\sqrt{\eta _1^2+\cdots +\eta _m^2}} \varphi _i$,
we have
\begin{align*}
\sum_{i=1}^m \eta _i(A_i-I_i)
&=\sqrt{\eta _1^2+\cdots +\eta _m^2}
\Big(\int _D f w \, dx-\int _D g(\sqrt{\eta _1^2+\cdots +\eta _m^2} w+U) w \, dx\Big) \\
& <-\epsilon  \sqrt{\eta _1^2+\cdots +\eta _m^2} \,,
\end{align*}
for some $\epsilon >0$, when $\sqrt{\eta _1^2+\cdots +\eta _m^2}$ is large,
in view of  condition \eqref{9.4}.
If we denote by $h$ an upper bound on all of $(A_i-I_i)^2$, then from \eqref{9.45}
\[
\sum_{i=1}^m \xi _i ^2<\sum_{i=1}^m \eta _i ^2-2\epsilon  \sqrt{\eta _1^2+
\cdots +\eta _m^2}+mh<\sum_{i=1}^m \eta _i ^2 \,,
\]
for  $\sqrt{\eta _1^2+\cdots +\eta _m^2}$ large.
Then the map $T$ is a compact and continuous map of a sufficiently large ball
in $R^m \times E^{\perp}$ into itself, and we have a solution by Schauder's
fixed point theorem.
\end{proof}

Condition \eqref{9.4} implies the Landesman and  Lazer \cite{L} condition \eqref{5}.
Indeed, if $\lambda _k$ is a simple eigenvalue, then $w=b \varphi _k$. If the constant
$b>0$ ($b<0$), then \eqref{9.4} implies the right (left) inequality in \eqref{5}.

In the ODE case a famous example of resonance with a two-dimensional eigen\-space
is the result of  Lazer and  Leach \cite{L1}, which we describe next.
We seek to find $2 \pi$ periodic solutions $u=u(t)$ of
\begin{equation} \label{10}
u''+n^2u +g(u)=f(t) \,,
\end{equation}
with a given continuous $2 \pi$ periodic $f(t)=f(t+2\pi)$, and $n$
is a positive integer. The linear part of this problem is at resonance, because
\[
u''+n^2u=0
\]
has a two-dimensional $2 \pi$ periodic null space, spanned by
$\cos nt$ and $\sin nt$. The following result is included in  Lazer and
 Leach \cite{L1}.

\begin{theorem}[\cite{L1}] \label{thm:2}
Assume that $g(u) \in C(R)$ satisfies \eqref{4}. Define
\begin{equation}
\label{***}
A=\int_0^{2\pi} f(t) \cos nt \, dt, \quad B=\int_0^{2\pi} f(t) \sin nt \, dt \,.
\end{equation}
Then a $2 \pi$ periodic solution of \eqref{10} exists if and only if
\begin{equation} \label{11}
\sqrt{A^2+B^2}<2\left(g(\infty)-g(-\infty) \right) \,.
\end{equation}
\end{theorem}

The following elementary lemma is easy to prove.

\begin{lemma}\label{lma:2}
Consider a function $\cos (nt-\delta)$, with an integer $n$ and any real $\delta$.
Denote $P=\{t \in (0,2\pi) : \cos (nt-\delta)>0 \}$ and
 $N=\{t \in (0,2\pi) : \cos (nt-\delta)<0 \}$. Then
\[
\int_P \cos (nt-\delta) \, dt=2, \quad \int_N \cos (nt-\delta) \, dt=-2 \,.
\]
\end{lemma}

We show next that Theorem  \ref{thm:2} of   Lazer and  Leach provides
an example to Theorem \ref{thm:10} of  Williams.
Indeed, any eigenfunction corresponding to $\lambda _n =n^2$ is of the form
$w=a \cos nt +b \sin nt$, $a,b \in R$, or $w =\sqrt{a^2+b^2} \cos (nt-\delta)$
for some $\delta$. The left hand side of \eqref{9.4} is $aA+bB$, while the
integral on the right is equal to
$2 \sqrt{a^2+b^2} \left(g(\infty)-g(-\infty) \right)$, in view of
Lemma \ref{lma:2}. We then rewrite \eqref{9.4} in terms of a scalar product
of two vectors
\[
\Big(\frac{a}{\sqrt{a^2+b^2}}, \frac{b}{\sqrt{a^2+b^2}} \Big) \cdot (A,B)<2
 \left(g(\infty)-g(-\infty) \right) \quad \text{for all $a$ and $b$}\,,
\]
which is equivalent to the  Lazer and Leach condition \eqref{11}.

Another perturbation of a harmonic oscillator at resonance was considered
by  Lazer and Frederickson \cite{FL}, and  Lazer \cite{L2}:
\begin{equation} \label{f1}
u''+g(u)u'+n^2u=f(t) \,.
\end{equation}
Here $f(t) \in C(R)$ satisfies $f(t+2\pi)=f(t)$ for all $t$,
$g(u) \in C(R)$, $n \geq 1$ is an integer.
Define $G(u)=\int_0^u g(t) \, dt$. We assume that the finite limits
$G(\infty)$ and  $G(-\infty)$ exist, and
\begin{equation} \label{f2}
G(-\infty)<G(u)<G(\infty)  \quad \text{for all } u \,.
\end{equation}

\begin{theorem}\label{thm:***}
Assume that \eqref{f2} holds, and let $A$ and $B$ be again defined by \eqref{***}.
Then the condition
\begin{equation} \label{f3}
\sqrt{A^2+B^2}<2n \left(G(\infty)-G(-\infty) \right)
\end{equation}
is necessary and sufficient for the existence of $2\pi$ periodic solution
of \eqref{f1}.
\end{theorem}

This result was proved in  Lazer \cite{L2} for $n=1$, and by  Korman and
 Li \cite{KL}, for $n \geq 1$. Observe that the condition for solvability
now depends on $n$, unlike the  Lazer and Leach condition \eqref{11}.
Also,  Theorem \ref{thm:***} does not carry over to boundary value problems,
see \cite{KL}, unlike the result of Lazer and  Leach.

Next, we discuss the result of  de Figueiredo and  Ni \cite{FN},
involving resonance at the principal eigenvalue.

\begin{theorem}[\cite{FN}] \label{thm:4}
Consider the problem
\begin{equation} \label{23}
\Delta u+\lambda_1 u+g(u)=e(x), \quad x\in D, \quad u=0 \quad \text{on }
\partial D \,.
\end{equation}
Assume that $e(x) \in L^{\infty}(D)$ satisfies $\int _D e(x) \varphi _1(x) \, dx=0$,
while the function $g(u) \in C(R)$  is a bounded function, satisfying
\begin{equation} \label{24}
ug(u)>0, \quad \text{for all $u \in R$} \,.
\end{equation}
Then the problem \eqref{23} has a solution $u(x) \in W^{2,p}(D) \cap W_0^{1,p}(D)$,
for any $p>1$.
\end{theorem}

If, in addition to \eqref{24}, we have
\begin{equation} \label{25}
\liminf _{u \to \infty} g(u)>0, \quad \text{and} \quad
\limsup _{u \to -\infty} g(u)<0 \,,
\end{equation}
then the previous Theorem \ref{thm:3} applies.
The result of  de Figueiredo and  Ni \cite{FN} allows either one (or both)
of these limits to be zero.

\begin{proof}[Proof of  Theorem \ref{thm:4}]
 Again we follow the proof of  Theorem \ref{thm:3}.  As before, we set up
the map $T:  (\eta _1,V) \to (\xi _1,U)$, taking the space $R \times \varphi _1^{\perp}$
into itself. We use \eqref{8}, with $k=1$, to compute $U$, while equation
\eqref{8.1} takes the form
\begin{equation} \label{26}
\xi _1=\eta _1-\int _D g(\eta _1 \varphi _1 (x)+U(x)) \varphi _1(x) \, dx\,.
\end{equation}
Since $\|U\|_{L^{\infty}(D)}$ is bounded uniformly in $(\eta _1,V)$, we can
find $M>0$, so that for all $x \in D$
\begin{gather*}
 \xi _1 \varphi _1 (x)+U(x)>0, \quad \text{for $\xi _1>M$} \\
 \xi _1 \varphi _1 (x)+U(x)<0, \quad \text{for $\xi _1<-M$} \,.
\end{gather*}
Hence, $\xi _1<\eta _1$ for $\eta _1>0$ and large, while $\xi _1>\eta _1$
for $\eta _1<0$ and $|\eta _1|$ large. As before, the map
$T: (\eta _1,V) \to (\xi _1,U)$ is a continuous and compact map, taking
a sufficiently large ball of $R \times \varphi _1^{\perp}$ into itself,
and the proof follows  by Schauder's fixed point theorem.
\end{proof}

This result was generalized to  unbounded $g(u)$ by  Iannacci,  Nkashama and
 Ward \cite{INW}. We now extend Theorem \ref{thm:4} to the higher eigenvalues.

\begin{theorem}
Consider the problem
\begin{equation} \label{26.1}
\Delta u+\lambda_k u+g(u)=e(x), \quad x \in D, \quad u=0 \quad \text{on } \partial D \,.
\end{equation}
Assume that $\lambda _k$, $ k \geq 1$,  is a simple eigenvalue, and
$e(x) \in L^{\infty}(D)$ satisfies
 $\int _D e(x) \varphi _k(x) \, dx=0$, while the function $g(u) \in C(R)$
is a bounded function, satisfying \eqref{24} and \eqref{25}.
Then the problem \eqref{26.1} has a solution
$u(x) \in W^{2,p}(D) \cap W_0^{1,p}(D)$, for any $p>1$.
\end{theorem}

\begin{proof}
We follow the proof of Theorem \ref{thm:3}, replacing  problem \eqref{26.1}
by its Lyapunov-Schmidt decomposition \eqref{6}, \eqref{7}, then setting
up the map $T$, given by  \eqref{8} and
\[
\xi _k=\eta _k-\int _D g(\eta _k \varphi _k (x)+U(x)) \varphi _k(x) \, dx\,.
\]
Since $g(u)$ is bounded, the $L^{\infty}$ estimate \eqref{9} holds.
By our conditions on $g(u)$, $\xi _k <\eta _k$ ($\xi _k >\eta _k$),
provided that $\eta _k>0$ ($\eta _k<0$) and $|\eta _k|$ is large.
As in the proof of Theorem \ref{thm:3}, we conclude that the map $T$
has a fixed point.
\end{proof}

We now review another extension  of the result of  Iannacci et al \cite{INW}
to  the problem
\begin{equation} \label{230}
\Delta u +\lambda _1 u+g(u)=\mu _1 \varphi _1+e(x) \quad \text{on }\Omega, \quad u=0 \quad
\text{on } \partial D \,,
\end{equation}
with $e(x) \in \varphi _1 ^\perp$. Decompose the solution $u(x)=\xi _1 \varphi _1+U$,
with $U \in \varphi _1^\perp$. We wish to find a solution pair
$(u, \mu _1)=(u, \mu _1)(\xi _1)$, i.e., the global solution curve.
We proved the following result in  Korman \cite{K2}.

\begin{theorem}
Assume that $g(u) \in C^1(R)$ satisfies
\begin{gather*}
u g(u) >0 \quad \text{for all } u \in R \,, \\
 g'(u) \leq \gamma< \lambda _2-\lambda _1 \quad \text{for all } u \in R \,.
\end{gather*}
Then there is a continuous curve of solutions of \eqref{230}:
$(u(\xi _1),\mu _1(\xi _1))$, $u \in H^2(D) \cap H^1_0(D)$, with
$-\infty<\xi _1<\infty$, and $\int _D u(\xi _1) \varphi _1 \, dx=\xi _1$.
This curve exhausts the solution set of \eqref{230}. The continuous
function $\mu _1(\xi _1)$ is positive for $\xi _1 >0$ and large,
and $ \mu _1(\xi _1)<0$ for $\xi _1 <0$ and $|\xi _1|$ large.
In particular, $\mu _1(\xi^0 _1)=0$ at some $\xi^0 _1$, i.e.,
 we have a solution of
\[
\Delta u +\lambda _1 u+g(u)=e(x) \quad \text{on }D, \quad u=0 \quad \text{on }
 \partial D \,.
\]
\end{theorem}

We see that the result of  Iannacci et al \cite{INW} corresponds to just
one point on this solution curve.

\section{Resonance for systems}

We begin by considering linear systems of the type
\begin{equation} \label{30}
\begin{gathered}
 \Delta u+au+bv=f(x) \quad x \in D, \quad u=0 \quad \text{on } \partial D \\
 \Delta v+cu+dv=g(x) \quad x \in D, \quad v=0 \quad \text{on } \partial D \,,
\end{gathered}
\end{equation}
with given functions $f(x)$ and $g(x)$, and constants $a$, $b$, $c$ and $d$.
As before, we denote by $\lambda _k$ the eigenvalues of $-\Delta$ with the
 Dirichlet boundary conditions (see \eqref{dl}),
and by $\varphi _k(x)$ the corresponding eigenfunctions. We shall assume throughout
 this section that   $\int _D \varphi ^2_k(x) \, dx=1$.

\begin{lemma}\label{lma:30}
Assume that
\[
(a-\lambda _k)(d-\lambda _k)-bc \ne 0, \quad  \text{for all $k \geq 1$} \,,
\]
i.e., all eigenvalues of $-\Delta$ are  not equal to the eigenvalues of the matrix
$\begin{bmatrix}
a & b\\
 c & d
\end{bmatrix}$.
Then for any pair $(f,g) \in L^2(D) \times L^2(D)$ there exists a unique
solution $(u,v)$, with $u$ and $v \in  W^{2,2}(D) \cap W_0^{1,2}(D) $.
Moreover, for some $c>0$
\[
\|u\|_{W^{2,2}(D)}+\|v\|_{W^{2,2}(D)}
\leq c \left( \|f\|_{L^2(D)}+\|g\|_{L^2(D)} \right) \,.
\]
\end{lemma}

\begin{proof}
Using Fourier series, $f(x)=\Sigma _{k=1}^{\infty} f_k \varphi _k$,
$g(x)=\Sigma _{k=1}^{\infty} g_k \varphi _k$, $u(x)=\Sigma _{k=1}^{\infty} u_k \varphi _k$
 and $v(x)=\Sigma _{k=1}^{\infty} v_k \varphi _k$, we obtain the unique
solution $(u,v) \in L^2(D) \times L^2(D)$. Using  elliptic estimates
for each equation separately, we obtain
\begin{equation} \label{31}
\|u\|_{W^{2,2}(D)}+\|v\|_{W^{2,2}(D)}
\leq c \left( \|u\|_{L^2(D)}+\|v\|_{L^2(D)}+\|f\|_{L^2(D)}+\|g\|_{L^2(D)} \right) \,.
\end{equation}
Since we have uniqueness of solution for \eqref{30}, the extra terms
$\|u\|_{L^2(D)}$ and $\|v\|_{L^2(D)}$ are removed in a standard way.
\end{proof}

The following two lemmas are proved similarly.

\begin{lemma}\label{lma:31}
Consider the problem (here $\mu$ is a constant)
\begin{equation} \label{32}
\begin{gathered}
 \Delta u+\lambda _k u=f(x) \quad x\in D, \quad u=0 \quad \text{on } \partial D \\
 \Delta v+\mu v=g(x) \quad x\in D, \quad v=0 \quad \text{on } \partial D \,,
\end{gathered}
\end{equation}
Assume that $\mu \ne \lambda _n$, for all $n \geq 1$, $f(x) \in \varphi ^{\perp}_k$,
$g(x) \in L^2(D)$. One can select a solution such that $u(x) \in \varphi ^{\perp}_k$,
and $v(x) \in L^2(D)$. Such a solution is unique, and the estimate \eqref{31} holds.
\end{lemma}

\begin{lemma}\label{lma:32}
Consider the problem
\begin{equation} \label{33}
\begin{gathered}
\Delta u+\lambda _k u=f(x) \quad x\in D, \quad u=0 \quad \text{on } \partial D \\
\Delta v+\lambda _m v=g(x) \quad x\in D, \quad v=0 \quad \text{on } \partial D \,,
\end{gathered}
\end{equation}
Assume that  $f(x) \in \varphi ^{\perp}_k$, $g(x) \in \varphi ^{\perp}_m$.
 One can select a solution such that $u(x) \in \varphi ^{\perp}_k$, and
$v(x) \in \varphi ^{\perp}_m$. Such a solution is unique, and the
estimate \eqref{31} holds.
\end{lemma}

\begin{lemma}\label{lma:33}
Consider the problem
\begin{equation} \label{34}
\begin{gathered}
 \Delta u+\lambda _k u+v=f(x) \quad x\in D, \quad u=0 \quad \text{on } \partial D \\
 \Delta v+\lambda _k v=g(x) \quad x\in D, \quad v=0 \quad \text{on } \partial D \,,
\end{gathered}
\end{equation}a
Assume that  $f(x) \in  \varphi ^{\perp}_k$, $g(x) \in \varphi ^{\perp}_k$.
One can select a solution such that $u(x) \in \varphi ^{\perp}_k$, and
$v(x) \in \varphi ^{\perp}_k$. Such a solution is unique, and the estimate
\eqref{31} holds.
\end{lemma}

\begin{proof}
The second equation in \eqref{34} has infinitely many solutions of the
form $v=v_0+c \varphi _k$. We now select $c$, so that  $v \in \varphi ^{\perp}_k$.
The first equation then takes the form
\[
\Delta u+\lambda _k u=f-v_0-c \varphi _k \in \varphi ^{\perp}_k \,.
\]
This equation has infinitely many solutions of the form
$u=u_0+c_1 \varphi _k$. We  select $c_1$, so that  $u \in \varphi ^{\perp}_k$.
\end{proof}

We now turn to nonlinear systems
\begin{equation} \label{35}
\begin{gathered}
 \Delta u+au+bv+f(u,v)=h(x) \quad x\in D, \quad u=0 \quad \text{on } \partial D \\
 \Delta v+cu+dv+g(u,v)=k(x) \quad x\in D, \quad v=0 \quad \text{on } \partial D \,,
\end{gathered}
\end{equation}a
with given $h(x)$, $k(x)$ and bounded $f(u,v)$, $g(u,v)$.
If there is no resonance, existence of solutions is easy.

\begin{theorem}\label{thm:31}
Assume that $(a-\lambda _n)(d-\lambda _n)-bc \ne 0$, for all $n \geq 1$, and
$f(u,v)$, $g(u,v)$ are bounded.
Then for any pair $(h,k) \in L^2(D) \times L^2(D)$ there exists a  solution
$(u,v)$, with $u$ and $v \in  W^{2,2}(D) \cap W_0^{1,2}(D) $.
\end{theorem}

\begin{proof}
The map $(w,z) \to (u,v)$, obtained by solving
\begin{gather*}
 \Delta u+au+bv=h(x)-f(w,z) \quad x\in D, \quad u=0 \quad \text{on } \partial D \\
 \Delta v+cu+dv=k(x)-g(w,z) \quad x\in D, \quad v=0 \quad \text{on } \partial D \,,
\end{gather*}
in view of Lemma \ref{lma:30} (and Sobolev's embedding), is a compact and
continuous map of a sufficiently large ball around the origin in
$L^2(D) \times L^2(D)$ into itself, and Schauder's fixed point theorem applies.
\end{proof}

Next,  we discuss the case of resonance, when  one of the eigenvalues of
 the coefficient matrix
$A=\begin{bmatrix}
a & b\\
 c & d
\end{bmatrix}$  is $\lambda  _k$.
We distinguish the following four possibilities: the second eigenvalue
of $A$ is not equal to $\lambda _m$ for all $m \geq 1$, the second eigenvalue
of $A$ is equal to $\lambda _m$ for some  $m \ne k$, the second eigenvalue
of $A$ is  equal to $\lambda _k$, and the  matrix $A$ is diagonalizable,
and finally the second eigenvalue of $A$ is  equal to $\lambda _k$, and
the matrix $A$ is not diagonalizable.
By a linear change of variables, $(u,v) \to (u_1,v_1)$,
$\begin{bmatrix}
u \\
 v
\end{bmatrix}
=Q \begin{bmatrix}
u_1 \\
 v_1
\end{bmatrix}$,
with a non-singular matrix $Q$, we transform the system \eqref{35} into
\[
\Delta
\begin{bmatrix}
u_1 \\
 v_1
\end{bmatrix}
 +Q^{-1}AQ \begin{bmatrix}
u_1 \\
 v_1
\end{bmatrix}
=Q^{-1}\begin{bmatrix}
h \\
 k
\end{bmatrix}
\]
with the matrix $Q^{-1}AQ$ being either diagonal, or the Jordan block
$\begin{bmatrix}
\lambda _k & 1\\
 0 & \lambda _k
\end{bmatrix}$.
Let us assume that this change of variables has been performed, so that
there are three canonical cases to consider.

We consider the system
\begin{equation} \label{36}
\begin{gathered}
 \Delta u+\lambda _k u+f(u,v)=h(x) \quad x\in D, \quad u=0 \quad \text{on } \partial D \\
 \Delta v+\mu v+g(u,v)=k(x) \quad x\in D, \quad v=0 \quad \text{on } \partial D \,.
\end{gathered}
\end{equation}
We assume that $f(u,v)$, $g(u,v) \in C(R \times R)$ are bounded on $R\times R$,
and there exist numbers $c$, $d$, $C$ and $D$, with  $c<d$ and $C<D$, such that
\begin{gather} \label{37}
f(u,v) > D \quad \text{for $u > d$}, \quad \text{uniformly in }v \in R \,, \\
 \label{38}
f(u,v) < C \quad \text{for } u < c, \quad \text{uniformly in }v \in R \,.
\end{gather}
Define
\[
L_2=D \int _{\varphi _k>0} \varphi _k \, dx+C \int _{\varphi _k<0} \varphi _k \, dx,
\quad L_1=C \int _{\varphi _k>0} \varphi _k \, dx+D \int _{\varphi _k<0} \varphi _k \, dx \,.
\]
Observe that $L_2>L_1$, because $D>C$.

\begin{theorem}\label{thm:30}
Assume that $\lambda _k$ is a simple eigenvalue, $\mu \ne \lambda _n$ for all
$n \geq 1$, while $f(u,v)$, $g(u,v) \in C(R \times R)$ are bounded on $R\times R$,
and  satisfy \eqref{37} and \eqref{38}. Assume that $h(x)$, $k(x) \in L^p(D)$,
for some  $p>n$. Assume finally that
\begin{equation} \label{39}
L_1<\int _D h(x) \varphi _k(x) \, dx<L_2 \,.
\end{equation}
Then the problem \eqref{36} has a solution $(u,v)$, with
$u,v \in  W^{2,p}(D) \cap W_0^{1,p}(D) $.
\end{theorem}

\begin{proof}
Denoting $A_k=\int _D h(x) \varphi _k \, dx$, we decompose
$h(x)=A_k \varphi _k (x)+e(x)$, with $e(x) \in \varphi _k^{\perp}$
(where $\varphi _k^{\perp}$ is a subspace of $L^2(D)$).
Similarly, we decompose the solution $u(x)=\xi _k \varphi _k (x)+U(x)$,
with $U(x) \in \varphi _k^{\perp}$.
Multiply the first equation in \eqref{36} by $\varphi _k$, and integrate
\begin{equation} \label{40}
\int _D f(\xi _k \varphi _k (x)+U(x),v) \varphi _k(x) \, dx=A_k \,.
\end{equation}
Then the first equation in \eqref{36} becomes
\begin{equation} \label{41}
\begin{gathered}
 \Delta U+\lambda_k U=-f(\xi _k \varphi _k +U,v)+\varphi _k \int _D f(\xi _k \varphi _k +U ,v) \varphi _k\, dx
 +e, \, x\in D \\
  U=0 \quad \text{on } \partial D \,.
\end{gathered}
\end{equation}
Equations \eqref{40} and \eqref{41} constitute the classical Lyapunov-Schmidt
reduction of the first equation in \eqref{36}. To solve \eqref{40}, \eqref{41}
and the second equation in \eqref{36}, we set up a map
 $T:  (\alpha _k,W,Z) \to (\xi _k,U,V)$, taking the space
$R \times \varphi _k^{\perp} \times L^2(D)$ into itself, by solving
(separately) the linear equations
\begin{equation} \label{42}
\begin{gathered}
 \Delta U+\lambda_k U=-f(\alpha _k \varphi _k +W,Z)+\varphi _k \int _D f(\alpha _k \varphi _k +W ,Z)
 \varphi _k\, dx+e,  \\
 \Delta V+\mu V=-g(\alpha _k \varphi _k +W,Z)+k(x) \\
    U=V=0 \quad \text{on } \partial D \,,
\end{gathered}
\end{equation}
and then setting
\begin{equation} \label{43}
\xi _k=\alpha _k +A_k- \int _D f(\alpha _k \varphi _k +U ,V) \varphi _k(x) \, dx \,.
\end{equation}

 By Lemma \ref{lma:31}, the map $T$ is well defined, and  $\|U\|_{W^{2,p}(D)}$
and $\|V\|_{W^{2,p}(D)}$ are bounded, and then by the Sobolev embedding
$\|U\|_{C^1(D)}$  and $\|V\|_{C^1(D)}$ are bounded uniformly in $(\alpha _k,W,Z)$.
This implies that if $\alpha _k$ is large  and positive, the integral in \eqref{43}
is greater than  $L_2>A_k$. When $\alpha _k$ is large in absolute value and negative,
the integral in \eqref{43} is smaller than $L_1<A_k$.
  It follows that for $\alpha _k$ large and positive,  $\xi _k<\alpha _k$,
 while for $\alpha _k$ large in absolute value and negative,  $\xi _k>\alpha _k$.
Hence, we can find a large $N$, so that if $\alpha _k \in (-N,N)$,
then  $\xi _k \in (-N,N)$. The map $T$ is a continuous and compact map,
taking a sufficiently large ball of $R \times \varphi _k^{\perp} \times L^2(D)$
into itself. By Schauder's fixed point theorem it has a fixed point,
which gives us a solution of \eqref{36}.
\end{proof}

We consider  next  the system
\begin{equation} \label{46}
\begin{gathered}
 \Delta u+\lambda _k u+f(u,v)=h(x) \quad x\in D, \quad u=0 \quad \text{on } \partial D \\
 \Delta v+\lambda _ m v+g(u,v)=k(x) \quad x\in D, \quad v=0 \quad \text{on } \partial D \,,
\end{gathered}
\end{equation}a
which includes, in particular, the case $m=k$.
Assume that there exist numbers $c_1$, $d_1$, $C_1$ and $D_1$, with
$c_1<d_1$ and $C_1<D_1$, such that
\begin{gather} \label{47}
g(u,v) > D_1 \quad \text{for }v > d_1, \quad \text{uniformly in } u \in R\,, \\
\label{48}
g(u,v) < C_1 \quad \text{for } v < c_1, \quad \text{uniformly in } u \in R \,.
\end{gather}
Define
\begin{gather*}
M_2=D_1 \int _{\varphi _m>0} \varphi _m \, dx+C_1 \int _{\varphi _m<0} \varphi _m \, dx, \\
M_1=C_1 \int _{\varphi _m>0} \varphi _m \, dx+D_1 \int _{\varphi _m<0} \varphi _m \, dx \,.
\end{gather*}
Observe that $M_2>M_1$, because $D_1>C_1$.

\begin{theorem}\label{thm:5}
Assume  $\lambda _k$ and $\lambda _m$ are simple eigenvalues, while
$f(u,v)$, $g(u,v)$ belong to $C(R \times R)$ are bounded on $R\times R$, and
satisfy \eqref{37}, \eqref{38} and \eqref{47},  \eqref{48}. Assume that
$h(x)$, $k(x) \in L^p(D)$, for some  $p>n$. Assume finally that
\begin{equation} \label{49}
L_1<\int _D h(x) \varphi _k(x) \, dx<L_2, \quad \text{and} \quad
M_1<\int _D k(x) \varphi _m (x)\, dx<M_2 \,.
\end{equation}
Then  problem \eqref{46} has a solution $(u,v)$, with
$u,v \in  W^{2,p}(D) \cap W_0^{1,p}(D) $.
\end{theorem}

\begin{proof}
Denoting $A_k=\int _D h(x) \varphi _k \, dx$ and
$B_m=\int _D k(x) \varphi _m \, dx$, we decompose
$h(x)=A_k \varphi _k (x)+e_1(x)$ and $k(x)=B_m \varphi _m (x)+e_2(x)$,
with $e_1(x) \in \varphi _k^{\perp}$ and $e_2(x) \in \varphi _m^{\perp}$.
Similarly, we decompose the solution $u(x)=\xi _k \varphi _k (x)+U(x)$
and $v(x)=\eta _m \varphi _m (x)+V(x)$, with $U(x) \in \varphi _k^{\perp}$
and $V(x) \in \varphi _m^{\perp}$. As before, we write down the Lyapunov-Schmidt
reduction of our problem \eqref{46}
\begin{gather*}
\int _D f(\xi _k \varphi _k +U ,\eta _m \varphi _m+V) \varphi _k(x) \, dx=A_k \\
 \int _D g(\xi _k \varphi _k +U ,\eta _m \varphi _m+V) \varphi _m(x) \, dx=B_m \\
\begin{aligned}
\Delta U+\lambda_k U
&=-f(\xi _k \varphi _k +U,\eta _m \varphi _m+V)\\
&\quad +\varphi _k \int _D f(\xi _k \varphi _k +U ,\eta _m \varphi _m+V) \varphi _k\, dx+e_1,
\end{aligned} \\
\begin{aligned}
\Delta V+\lambda_m V
&=-g(\xi _k \varphi _k +U,\eta _m \varphi _m+V) \\
&\quad +\varphi _m \int _D g(\xi _k \varphi _k +U ,\eta _m \varphi _m+V) \varphi _m\, dx+e_2,
\end{aligned} \\
  U=V=0 \quad \text{on } \partial D \,.
\end{gather*}
To solve this system, we set up a map
$T:  (\alpha _k,W,\beta _m,Z) \to (\xi _k,U,\eta _m,V)$,
taking the space $(R \times \varphi _k^{\perp}) \times (R \times \varphi_m ^{\perp})$
into itself, by solving (separately) the linear problems
\begin{gather*}
\begin{aligned}
\Delta U+\lambda_k U
&=-f(\alpha _k \varphi _k +W,\beta _m \varphi _m+Z)  \\
&\quad +\varphi _k \int _D f(\alpha _k \varphi _k +W ,\beta _m \varphi _m+Z) \varphi _k\, dx+e_1,
\end{aligned} \\
\begin{aligned}
 \Delta V+\lambda _m  V
&=-g(\alpha _k \varphi _k +W,\beta _m \varphi _m+Z) \\
&\quad  +\varphi _m \int _D g(\alpha _k \varphi _k +W ,\beta _m \varphi _m+Z) \varphi _m\, dx+e_2,
\end{aligned}\\
   U=V=0 \quad \text{on } \partial D \,,
\end{gather*}
and then setting
\begin{equation} \label{55}
\begin{gathered}
 \xi _k=\alpha _k +A_k- \int _D f(\alpha _k \varphi _k +U ,\beta _m \varphi _m+V) \varphi _k(x) \, dx \\
 \eta _m=\beta _m +B_m -\int _D g(\alpha _k \varphi _k +U ,\beta _m \varphi _m+V) \varphi _m\, dx \,.
\end{gathered}
\end{equation}

By Lemma \ref{lma:32}, the map $T$ is well defined, and
$\|U\|_{C^1(D)}$  and $\|V\|_{C^1(D)}$ are bounded uniformly in
$(\alpha _k,W,\beta _m,Z)$.
This implies that if $\alpha _k$ is large  and positive,
$\int _D f(\alpha _k \varphi _k +U ,\beta _m \varphi _m+V) \varphi _k\, dx>L_2>A_k$.
When $\alpha _k$ is large in absolute value and negative, this integral
is smaller than $L_1<A_k$.
  It follows that for $\alpha _k$ large and positive,  $\xi _k<\alpha _k$,
 while for $\alpha _k$ large and negative,  $\xi _k>\alpha _k$. Hence, we
can find a large $N$, so that if $\alpha _k \in (-N,N)$,  then
$\xi _k \in (-N,N)$, and arguing similarly with the second line in \eqref{55},
 we see that if $\beta _m \in (-N,N)$,  then  $\eta _m \in (-N,N)$
(possibly with a larger $N$). The map $T$ is a continuous and compact map,
taking a sufficiently large ball of
$(R \times \varphi _k^{\perp}) \times (R \times \varphi_m ^{\perp})$ into itself.
 By Schauder's fixed point theorem it has a fixed point, which gives us
 a solution of \eqref{46}.
\end{proof}

We now turn to the final case
\begin{equation} \label{56}
\begin{gathered}
\Delta u+\lambda _k u+v+f(u,v)=h(x) \quad x\in D, \quad u=0 \quad \text{on } \partial D \\
\Delta v+\lambda _ k v+g(u,v)=k(x) \quad x\in D, \quad v=0 \quad \text{on } \partial D \,.
\end{gathered}
\end{equation}
Assume that there exist numbers $c_2$, $d_2$, $C_2$ and $D_2$, with
$c_2<d_2$ and $C_2<D_2$, such that
\begin{gather} \label{57}
g(u,v) > D_2 \quad \text{for } u > d_2, \quad \text{uniformly in }v \in R \,, \\
\label{58}
g(u,v) < C_2 \quad \text{for $u < c_2$}, \quad \text{uniformly in $v \in R$} \,.
\end{gather}
Define
\[
N_2=D_2 \int _{\varphi _k>0} \varphi _k \, dx+C_2 \int _{\varphi _k<0} \varphi _k \, dx, \quad
N_1=C_2 \int _{\varphi _k>0} \varphi _k \, dx+D_2 \int _{\varphi _k<0} \varphi _k \, dx \,.
\]
Observe that $N_2>N_1$, because $D_2>C_2$.

\begin{theorem}\label{thm:8}
Assume that $\lambda _k$ is a simple eigenvalue, while $f(u,v)$,
$g(u,v) \in C(R \times R)$ are bounded on $R\times R$, and  $g$
satisfies \eqref{57}, \eqref{58}. Assume that $h(x)$, $k(x) \in L^p(D)$,
for some  $p>n$. Assume finally that
\begin{equation} \label{59}
N_1<\int _D k(x) \varphi _k \, dx<N_2 \,.
\end{equation}
Then  problem \eqref{56} has a solution $(u,v)$, with
$u,v \in  W^{2,p}(D) \cap W_0^{1,p}(D) $.
\end{theorem}

\begin{proof}
Denoting $A_k=\int _D h(x) \varphi _k \, dx$ and $B_k=\int _D k(x) \varphi _k \, dx$,
we decompose $h(x)=A_k \varphi _k (x)+e_1(x)$ and $k(x)=B_k \varphi _k (x)+e_2(x)$,
with $e_1, e_2 \in \varphi _k^{\perp}$, and also
decompose the solution $u(x)=\xi _k \varphi _k (x)+U(x)$ and
$v(x)=\eta _k \varphi _k (x)+V(x)$, with $U,  V \in \varphi _k^{\perp}$.
The Lyapunov-Schmidt reduction of our problem \eqref{56} is
\begin{equation} \label{59.1}
\begin{gathered}
\eta _k+ \int _D f(\xi _k \varphi _k +U ,\eta _k \varphi _k+V) \varphi _k(x) \, dx=A_k \\
 \int _D g(\xi _k \varphi _k +U ,\eta _k \varphi _k+V) \varphi _k(x) \, dx=B_k
\\
\begin{aligned}
\Delta U+\lambda_k U+V
&=-f(\xi _k \varphi _k +U,\eta _k \varphi _k+V)\\
&\quad +\varphi _k \int _D f(\xi _k \varphi _k +U ,\eta _k \varphi _k+V) \varphi _k\, dx+e_1,
\end{aligned}\\
\begin{aligned}
\Delta V+\lambda_k V
&=-g(\xi _k \varphi _k +U,\eta _k \varphi _k+V) \\
&\quad +\varphi _k \int _D g(\xi _k \varphi _k +U ,\eta _k \varphi _k+V) \varphi _k\, dx+e_2,
\end{aligned}\\
 U=V=0 \quad \text{on } \partial D \,.
\end{gathered}
\end{equation}
To solve this system, we set up a map
$T:  (\alpha _k,W,\beta _k,Z) \to (\xi _k,U,\eta _k,V)$, taking the
space $(R \times \varphi _k^{\perp}) \times (R \times \varphi_k ^{\perp})$ into itself,
by solving  the linear system
\begin{gather*}
\begin{aligned}
\Delta U+\lambda_k U+V
&=-f(\alpha _k \varphi _k +W,\beta _k \varphi _k+Z)\\
&\quad +\varphi _k \int _D f(\alpha _k \varphi _k +W ,\beta _k \varphi _k+Z) \varphi _k\, dx+e_1,
\end{aligned} \\
\begin{aligned}
 \Delta V+\lambda _k  V
&=-g(\alpha _k \varphi _k +W,\beta _k \varphi _k+Z) \\
&\quad +\varphi _k \int _D g(\alpha _k \varphi _k +W ,\beta _k \varphi _k+Z) \varphi _k\, dx+e_2,
\end{aligned} \\
 U=V=0 \quad \text{on } \partial D \,,
\end{gather*}
and then setting
\begin{equation} \label{60}
\begin{gathered}
\xi _k=\alpha _k +B_k- \int _D g(\alpha _k \varphi _k +U ,\beta _k \varphi _k+V) \varphi _k(x) \, dx \\
\eta _k=A_k -\int _D f(\alpha _k \varphi _k +U ,\beta _k \varphi _k+V) \varphi _k\, dx \,.
\end{gathered}
\end{equation}

Fixed points of $T$ give us  solutions of \eqref{59.1}, and hence of \eqref{56}.
By Lemma \ref{lma:33}, the map $T$ is well defined, and
$\|U\|_{C^1(D)}$  and $\|V\|_{C^1(D)}$ are bounded uniformly in
 $(\alpha _k,W,\beta _k,Z)$.
This implies that if $\alpha _k$ is large  and positive,
$\int _D g(\alpha _k \varphi _k +U ,\beta _k \varphi _k+V) \varphi _k\, dx>N_2>B_k$.
When $\alpha _k$ is large in absolute value and negative, this integral
is smaller than $N_1<B_k$.
  It follows that for $\alpha _k$ large and positive,  $\xi _k<\alpha _k$,
while for $\alpha _k$ large and negative,  $\xi _k>\alpha _k$. Hence, we can
find a large $N$, so that if $\alpha _k \in (-N,N)$,  then  $\xi _k \in (-N,N)$.
Since the right hand side of the second line in \eqref{60} is bounded,
we see that if $\beta _k \in (-N,N)$,  then  $\eta _k \in (-N,N)$
(possibly with a larger $N$). The map $T$ is a continuous and compact map,
taking a sufficiently large ball of
$(R \times \varphi _k^{\perp}) \times (R \times \varphi_k ^{\perp})$ into itself.
By Schauder's fixed point theorem it has a fixed point, which gives us
 a solution of \eqref{56}.
\end{proof}

\section{Appendix: A direct proof of the theorem of  Lazer and  Leach}

Many proofs of this classical result are available, including the one above,
and a recent proof in the paper of  Hastings and  McLeod \cite{H},
which also has references to other proofs. In this appendix we present
a proof which is consistent with our approach in the present paper.

\begin{proof}[Proof of  Theorem \ref{thm:2}]
   Let $L^2_n=\{ u(t) \in L^2(R), u(t+2\pi)=u(t) \quad \text{for all }t :
 \int_0^{2 \pi} u(t) \cos nt \, dt=\int_0^{2 \pi} u(t) \sin nt \, dt=0\}$.
 As before, we decompose
\begin{equation} \label{14}
\begin{gathered}
 f(t)=\frac{A}{\pi} \cos nt+\frac{B}{\pi} \cos nt+e(t)  \\
 u(t)=\xi  \cos nt+\eta\cos nt+U(t) \,,
\end{gathered}
\end{equation}
with $e(t),U(t) \in L^2_n$. Multiply \eqref{10} by $\cos nt$
and $\sin nt$ respectively, and integrate
\begin{equation} \label{15}
\begin{gathered}
 \int_0^{2\pi} g \left(\xi  \cos nt+\eta\cos nt+U(t) \right) \cos nt \, dt=A \\
 \int_0^{2\pi} g \left(\xi  \cos nt+\eta\cos nt+U(t) \right) \sin nt \, dt=B \,.
\end{gathered}
\end{equation}
Using these equations, and the ansatz \eqref{14} in \eqref{10}
\begin{equation} \label{16}
\begin{aligned}
 U''+n^2U
&=-g \left(\xi  \cos nt+\eta\cos nt+U(t) \right) \\
&\quad  +\frac{1}{\pi} \Big( \int_0^{2\pi} g \left(\xi  \cos nt+\eta\cos nt+U(t)
\right) \cos nt \, dt \Big)  \cos nt \\
&\quad  +\frac{1}{\pi} \Big( \int_0^{2\pi} g \left(\xi  \cos nt+\eta\cos nt+U(t)
\right) \sin nt \, dt \Big) \sin nt +e(t) \,.
\end{aligned}
\end{equation}
Equations \eqref{15} and \eqref{16} provide the Lyapunov-Schmidt reduction
of \eqref{10}.

To prove the necessity part, we multiply the first equation in \eqref{15} by
$\frac{A}{\sqrt{A^2+B^2}}$, the second one by $\frac{B}{\sqrt{A^2+B^2}}$,
and add, putting the result in the form
\[
\int_0^{2\pi} g \left(\xi  \cos nt+\eta \sin nt+U(t) \right)
\cos (nt-\delta) \, dt=\sqrt{A^2+B^2} \,,
\]
for some $\delta$. Using Lemma \ref{lma:2}, the integral on the left is
 bounded from above by
\[
g(\infty)\int_P \cos (nt-\delta) \, dt +g(-\infty)\int_N \cos (nt-\delta)\, dt
=2\left(g(\infty)-g(-\infty) \right) \,.
\]
Turning to the sufficiency part, we set up a map $T: (a,b,V) \to (\xi,\eta,U)$,
 taking $R \times R \times L^2_n$ into itself, by solving
\begin{equation} \label{17}
\begin{aligned}
U''+n^2U &=-g \left(a  \cos nt+b\sin nt+V(t) \right) \\
&\quad +\frac{1}{\pi} \Big( \int_0^{2\pi} g \left(a  \cos nt+b \sin nt+V(t) \right)
\cos nt \, dt \Big)  \cos nt \\
&\quad  +\frac{1}{\pi} \Big( \int_0^{2\pi} g \left(a  \cos nt+b \sin nt+V(t) \right)
\sin nt \, dt \Big) \sin nt +e(t)
\end{aligned}
\end{equation}
for $U$, and then set
\begin{equation} \label{17a}
\begin{gathered}
\xi=a+A-\int_0^{2\pi} g \left(a  \cos nt+b \cos nt+U(t) \right) \cos nt \, dt \\
\eta=b+B-\int_0^{2\pi} g \left(a  \cos nt+b \cos nt+U(t) \right) \sin nt \, dt \,.
\end{gathered}
\end{equation}
 The right hand side of \eqref{17} is in $L^2_n$, and hence \eqref{17}
has infinitely many solutions of the form $U=U_0+c_1 \cos nt+c_2 \sin nt$.
We select the  unique pair $(c_1,c_2)$, so that $U \in L^2_n$. By the elliptic
theory, we then have (since $g(u)$ is bounded)
\begin{equation} \label{18}
\|U\|_{L^{\infty}} \leq c, \quad \text{with some $c>0$, uniformly in $ (a,b,V)$} \,.
\end{equation}
 We need to show that a sufficiently large ball in $(a,b)$ plane is mapped
into itself in $(\xi, \eta)$ plane.  We have
\[
a  \cos nt+b\sin nt=\sqrt{a^2+b^2} \cos (nt-\delta _1), \quad
 \text{for some $\delta _1$} \,.
\]
Then for $a^2+b^2$ large
\begin{align*}
& aA+bB-a \int_0^{2\pi} g \left(a  \cos nt+b \cos nt+U(t) \right) \cos nt \, dt\\
&-b \int_0^{2\pi} g \left(a  \cos nt+b \cos nt+U(t) \right) \sin nt \, dt \\
&\leq  \sqrt{a^2+b^2} \Big[ \sqrt{A^2+B^2}-\int_0^{2\pi} g
 \left(  \sqrt{a^2+b^2} \cos (nt-\delta _1)+U \right) \cos (nt-\delta _1) \, dt \Big] \\
& <-\mu \sqrt{a^2+b^2}, \quad \text{for some } \mu>0\,,
\end{align*}
because  the integral in the brackets on a sufficiently large ball gets
arbitrary close to $2\left(g(\infty)-g(-\infty) \right)>\sqrt{A^2+B^2}$.
Since $g(u)$ is bounded, we can find $h>0$ so that
\begin{gather*}
\Big( A-\int_0^{2\pi} g \left(a  \cos nt+b \cos nt+U(t) \right)
\cos nt \, dt \Big)^2 <h \,, \\
\Big(B-\int_0^{2\pi} g \left(a  \cos nt+b \cos nt+U(t) \right)
 \sin nt \, dt \Big)^2 <h \,.
\end{gather*}
Then from \eqref{17a}
\[
\xi ^2+\eta ^2<a^2+b^2-2\mu \sqrt{a^2+b^2}+2h<a^2+b^2 \,,
\]
for $a^2+b^2$ large.
Then the map $T$ is a compact and continuous map of a sufficiently
large ball in $R \times R \times L^2_n$ into itself, and we have a
solution by Schauder's fixed point theorem.
\end{proof}

\section{Appendix: Perturbations of forced harmonic oscillators at resonance
without Lazer-Leach condition}

We present next a result of de Figueiredo and  Ni's \cite{FN} type for
harmonic oscillators at resonance:
\begin{equation} \label{70}
u''+n^2 u+g(u)=e(t) \,.
\end{equation}

\begin{theorem}
Assume that $g(u) \in C(R)$ is a bounded function, and
\begin{gather}\label{71}
u g(u)>0 \quad \text{for all $u \in R$} \,, \\
\label{72}
\liminf _{u \to \infty} g(u)>0, \quad \limsup _{u \to -\infty} g(u)<0 \,.
\end{gather}
Assume that $e(t) \in C(R)$ is a $2 \pi$ periodic  function, satisfying
\begin{equation} \label{73}
\int _0^{2\pi} e(t) \sin nt \, dt=\int _0^{2\pi} e(t) \cos nt \, dt=0 \,.
\end{equation}
Then  problem \eqref{70} has a $2 \pi$ periodic solutions.
\end{theorem}

\begin{proof}
We follow the proof of Theorem \ref{thm:2}. As before,  equations \eqref{15},
with $A=B=0$,  and \eqref{16} provide the Lyapunov-Schmidt reduction
of \eqref{70}. To solve these equations, we again set up the map
$T: (a,b,V) \to (\xi,\eta,U)$, taking $R \times R \times L^2_n$ into itself,
 by solving \eqref{17} and then setting
\begin{equation} \label{74}
\begin{gathered}
\xi=a-\int_0^{2\pi} g \left(a  \cos nt+b \cos nt+U(t) \right) \cos nt \, dt
 \equiv a -I_1 \\
\eta=b-\int_0^{2\pi} g \left(a  \cos nt+b \cos nt+U(t) \right) \sin nt \, dt
 \equiv b -I_2\,.
\end{gathered}
\end{equation}
As before,
\begin{equation} \label{75}
\|U\|_{L^{\infty}} \leq c, \quad \text{with some $c>0$, uniformly in $ (a,b,V)$} \,.
\end{equation}
We have
\[
\xi ^2+\eta ^2 =a^2+b^2-2(aI_1+bI_2)+I^2_1+I^2_2 \,.
\]
Using the condition \eqref{72} and the estimate \eqref{75}, we can find a
constant $\mu>0$ such that
\[
aI_1+bI_2
=\sqrt{a^2+b^2} \int_0^{2\pi} g \Big(  \sqrt{a^2+b^2} \cos (nt-\delta _1)+U \Big)
\cos (nt-\delta _1) \, dt
\geq \mu \sqrt{a^2+b^2} \,,
\]
for $a^2+b^2$ large. Denoting by $h$ a bound on $I_1$ and $I_2$, we have
\[
\xi ^2+\eta ^2<a^2+b^2-2\mu \sqrt{a^2+b^2}+2h<a^2+b^2 \,,
\]
for $a^2+b^2$ large.
Then the map $T$ is a compact and continuous map of a sufficiently large
ball in $R \times R \times L^2_n$ into itself, and we have a solution by
Schauder's fixed point theorem.
\end{proof}

\end{document}